\documentclass[a4paper,12pt]{article}

\usepackage[psamsfonts]{amssymb}
\usepackage{amsmath,amsfonts}
\usepackage{amsthm}
\usepackage{graphicx}
\usepackage{graphicx}
\usepackage{breqn}

\newtheorem{theorem}{Theorem}
\newtheorem{proposition}{Proposition}
\newcommand{\num}{\text{Num}}

\newtheorem{lemma}{Lemma}

\date{July 1, 2018}

\theoremstyle{definition}
\newtheorem{definition}{Definition}

\title{Limiting curves for the dyadic odometer and the generalized Trollope-Delange formula.\footnote{The study of the limiting curves for the  dyadic odometer (Sec. 3.1) is supported by RFBR (grant 17-01-00433).  The study of the q-generalized Trollope-Delange formula (Sec.~3.2) is supported by the  Russian Science Foundation (grant 17-71-20153).}
}
\author{Aleksey~Minabutdinov\thanks{National Research University Higher School of Economics, Department of Applied Mathematics and Business Informatics, St.Petersburg, Russia, e-mail: \texttt{ aminabutdinov@hse.ru.}}}

\begin{document}
\maketitle

\begin{abstract}
We study limiting curves resulting from deviations in partial sums in the ergodic theorem for the dyadic odometer and non-cylindric functions.  In particular, we generalize the Trollope-Delange formula for the case of the weighted sum-of-binary-digits function and show that the Takagi-Landsberg curve arises.
\end{abstract}


%	Radix representation; digital problems
% Equations for real functions
% Ergodic theorems
\section{Introduction}

Let  $T$ be a measure preserving transformation defined on a Lebesgue probability space $(X, \mathcal{B},\mu)$ with an invariant ergodic probability measure $\mu$. Let    $g$ denote a function in $L^1(X,\mu)$. In \cite{DeLaRue} \'{E}. Janvresse,  {T.} de~la Rue, and {Y.}~Velenik in the process of studying the Pascal adic transformation\footnote{The Pascal adic transformation was invented by A.~Vershik (and independently by S.~Kakutani), see \cite{ver82}, \cite{Kakutani1976}, and intensively studied since then, see e.g. \cite{PascalLooselyBernoulli}, \cite{MelaPetersen}, \cite{ver11}, \cite{ver15}. According to the authors of \cite{DeLaRue} their research was motivated by the observation by X.~Mela.}  introduced a new notion of \emph{a limiting curve}. Following \cite{DeLaRue} for a point $x\in X$ and a positive integer $j$ we denote the partial sum $\sum\limits_{k=0}^{j-1}g\big(T^kx\big)$ by $\mathcal{S}_{x}^g(j)$. We extend the function $\mathcal{S}_{x}^g(j)$ to a real valued argument by a linear interpolation and denote extended function again by  $\mathcal{S}_{x}^g(j)$ or simply $\mathcal{S}(j),j\geq 0$.

Let  $(l_n)_{n=1}^{\infty}$ be a sequence of positive integers. We consider continuous on $[0,1]$ functions $\varphi_n(t) = \frac{\mathcal{S}(t\cdot l_n(x)) - t \cdot \mathcal{S}(l_n)}{R_{n}} \big( \equiv\varphi_{x,l_n}^g(t)  \big),$ where the normalizing coefficient $R_{n}$ is canonically defined to be equal to the maximum in $t\in[0,1]$ of $|\mathcal{S}(t\cdot l_n(x)) - t \cdot \mathcal{S}(l_n)|$.

\begin{definition}[\cite{DeLaRue}]\normalfont %{Определение 1.}
If there is a sequence $l^g_n(x)\in\mathbb{N}$  such that functions $\varphi_{x,l^g_n(x)}^g$ converge to a (continuous) function $\varphi_x^g$ in sup-metric on $[0,1],$ then
the graph of the limiting function $\varphi=\varphi^g_x$ is called a \emph{limiting curve}, sequence $l_n=l_n^g(x)$ is called a \emph{stabilizing sequence} and the sequence $R_{n} = R_{x,l^g_n(x)}^g$ is called a \emph{normalizing sequence}.  The quadruple $\Big(x, \big(l_n\big)_{n=1}^\infty, \big(R_{n}\big)_{n=1}^\infty, \varphi\Big)$ is called a \emph{limiting bridge}.
\label{def:LimitShape}
\end{definition}

Heuristically, the limiting curve describes   small fluctuations (of certainly renormalized) ergodic sums $\frac{1}{l}\mathcal{S}(l), l\in(l_n),$ along the forward trajectory $x, T(x), T^2(x)\dots$. More specifically, for $l\in(l_n)$ it holds $\mathcal{S}(t\cdot l) = t \mathcal{S}(l)+R_l \varphi(t)+o(R_l)$, where $t\in[0,1].$

For the Pascal adic transformation in \cite{DeLaRue} and \cite{LodMin16} it was shown  that for $\mu$-a.e. $x$ a limiting bridge exists and stabilizing sequence $l_n(x)$ can be chosen in a such way that the well-known Takagi curve (and its generalizations) arises in the limit. In \cite{Min17}  results were generalized for a wider class of polynomial adic dynamical systems. In general limiting curves are not well-studied yet for a lot of interesting transformations. In particular, authors of \cite{DeLaRue}  asked (see Sec. 4.3.2.) whether limiting curves can exist for transformations in the rank-one category.

Findings on limiting curves motivated several researches (see e.g. \cite{Nogueira})  to attack the problem
of the spectrum of the Pascal adic transformation. Despite interesting results were obtained, the problem is still
unsolved\footnote{However, there is a general consensus that spectrum should at least comprise a
 continuous component, see \cite{ver15}.}. A natural  starting  hypothesis is that for a system with
 discrete spectrum a (Besicovitch) almost periodic property of trajectories should imply that no limiting
  curve exists. However, this is not the case and the goal of this paper is to construct a concrete
   counterexample. %\footnote{Yet one can show that  uniform (or Bohr) almost periodicity is sufficient to exclude limiting curves.}
We consider the  dyadic odometer (as a simplest  rank-one system with discrete dyadic spectra) and show that for certain functions  limiting curves arise. Surprisingly,  limiting curves that we find belong to the so called Takagi class, see below. Technically, we generalize the famous Trollope-Delange formula for the case of weighted-sum-of-binary digits function, that does not seem to have been published before.

\section{Designations}

\subsection{Takagi-Landsberg functions}
Let $|a|<1$. The Takagi-Landsberg function with parameter $a$ is defined by the identity
\begin{equation}\mathcal{T}_a(x) = \sum\limits_{n=0}^\infty a^{n}\tau(2^n x),\label{eq:TakagiLandsberg}\end{equation}
 where
$\tau(x) = \text{dist}(x, \mathbb{Z}) $, the distance from $x$ to the nearest integer.  It is immediately clear that the series converges uniformly\footnote{More general $\sum\limits_{n=0}^\infty c_n\tau(2^n x)$ with $\sum\limits_{n=0}^\infty |c_n|<\infty$ can be considered; this family of functions also known as the Takagi class.}, and hence defines a continuous function $\mathcal{T}_a$ for $|a|<1$. Functions\footnote{For the Pascal adic another class of  generalized Takagi functions appeared in \cite{DeLaRue, LodMin16}, the only intersection is the Takagi function $\mathcal{T}$ itself.} $\{\mathcal{T}_a\}_a$ can be considered as  direct generalization of the famous Takagi-Blancmange function $\mathcal{T}$ which is obtained when $a$ equals to $1/2$, see  \cite{Takagi1903}.  Functions $\mathcal{T}_a$ are $1$-periodic and  nowhere differentiable for $|a|\geq\frac12$ but differentiable almost
everywhere\footnote{In particular, $\mathcal{T}_{1/4}(x) =x(1-x)$. } for $|a|<\frac12$, see \cite{Landsberg} and more results in \cite{Lagarias} and \cite{AllaartKawamura2012}. In \cite{Allaart11, Boros, TaborTabor2009} the so called approximate midconvexity property of the functions $\mathcal{T}_a, a\in[1/4,1/2]$, was studied.

\begin{figure}[t]
                \centering
                \includegraphics[scale=0.17]{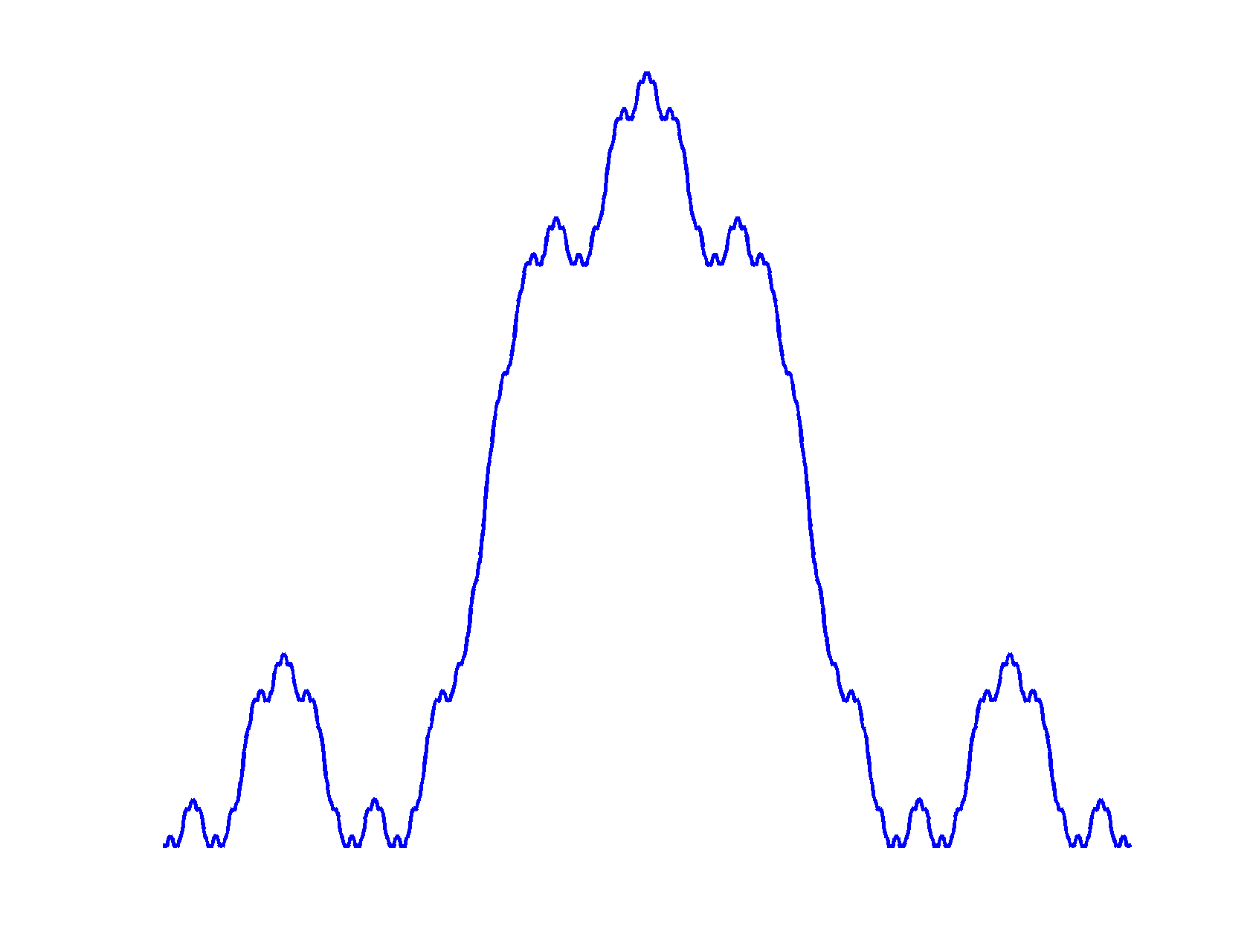}
                \includegraphics[scale=0.17]{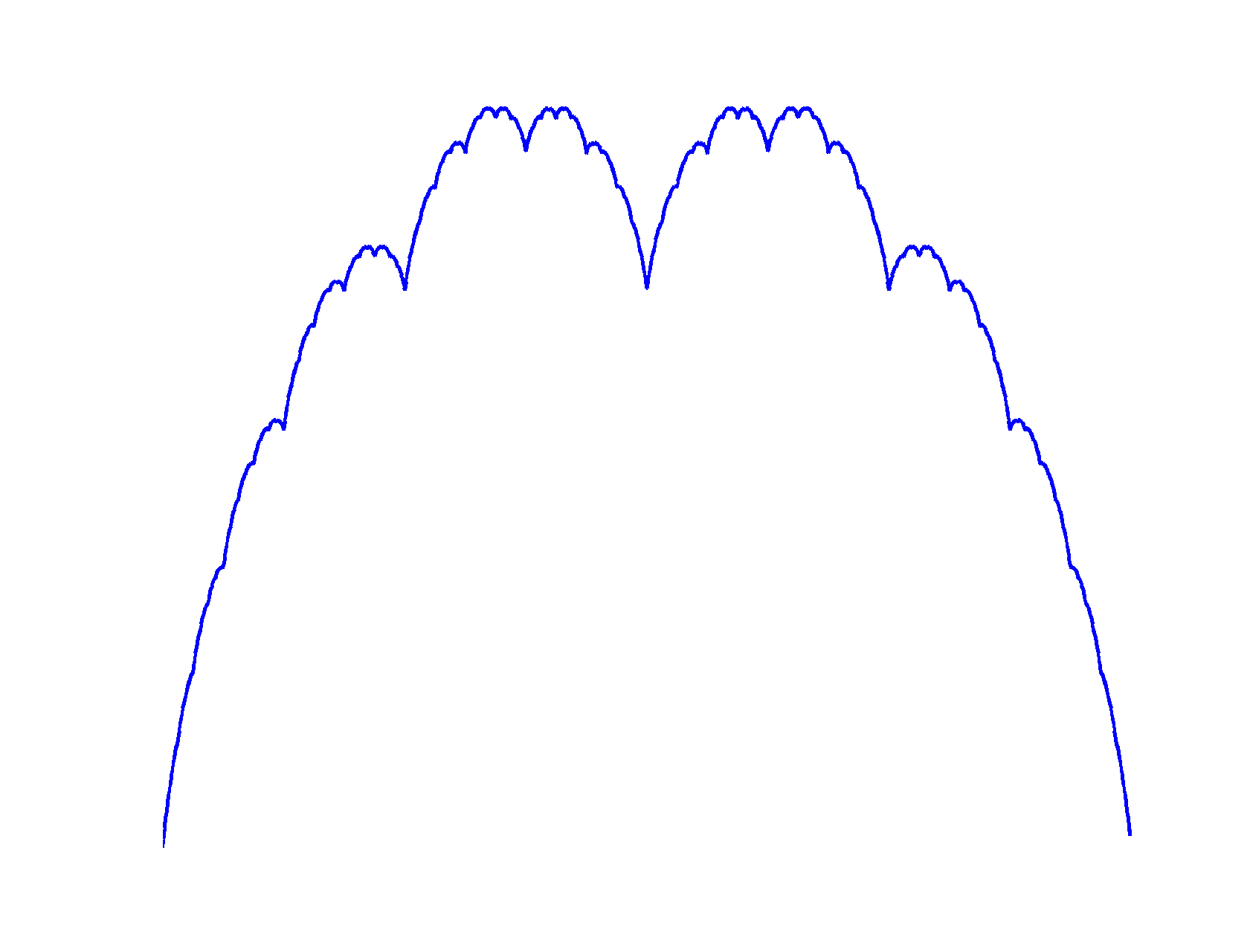}
                \includegraphics[scale=0.17]{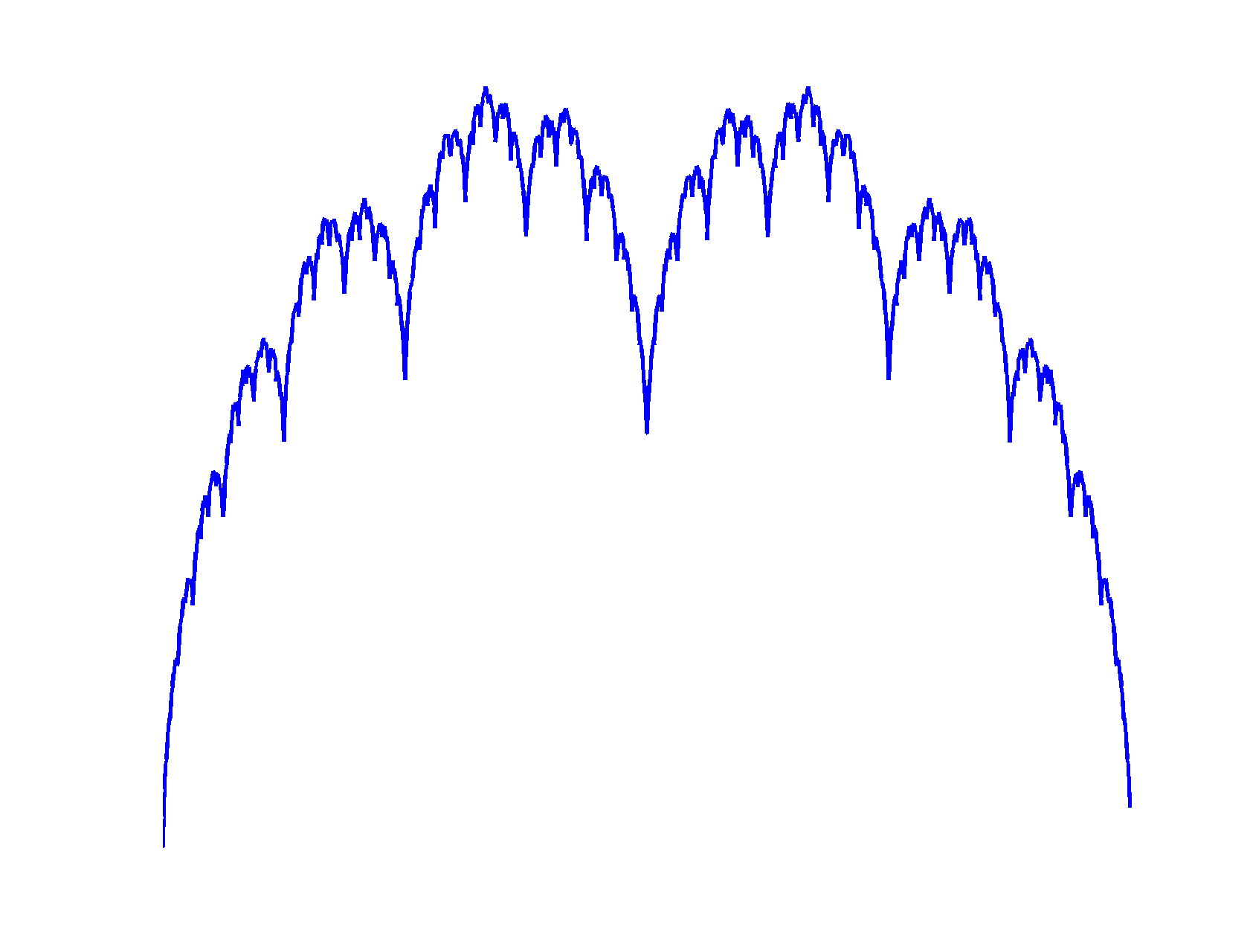}
                \includegraphics[scale=0.17]{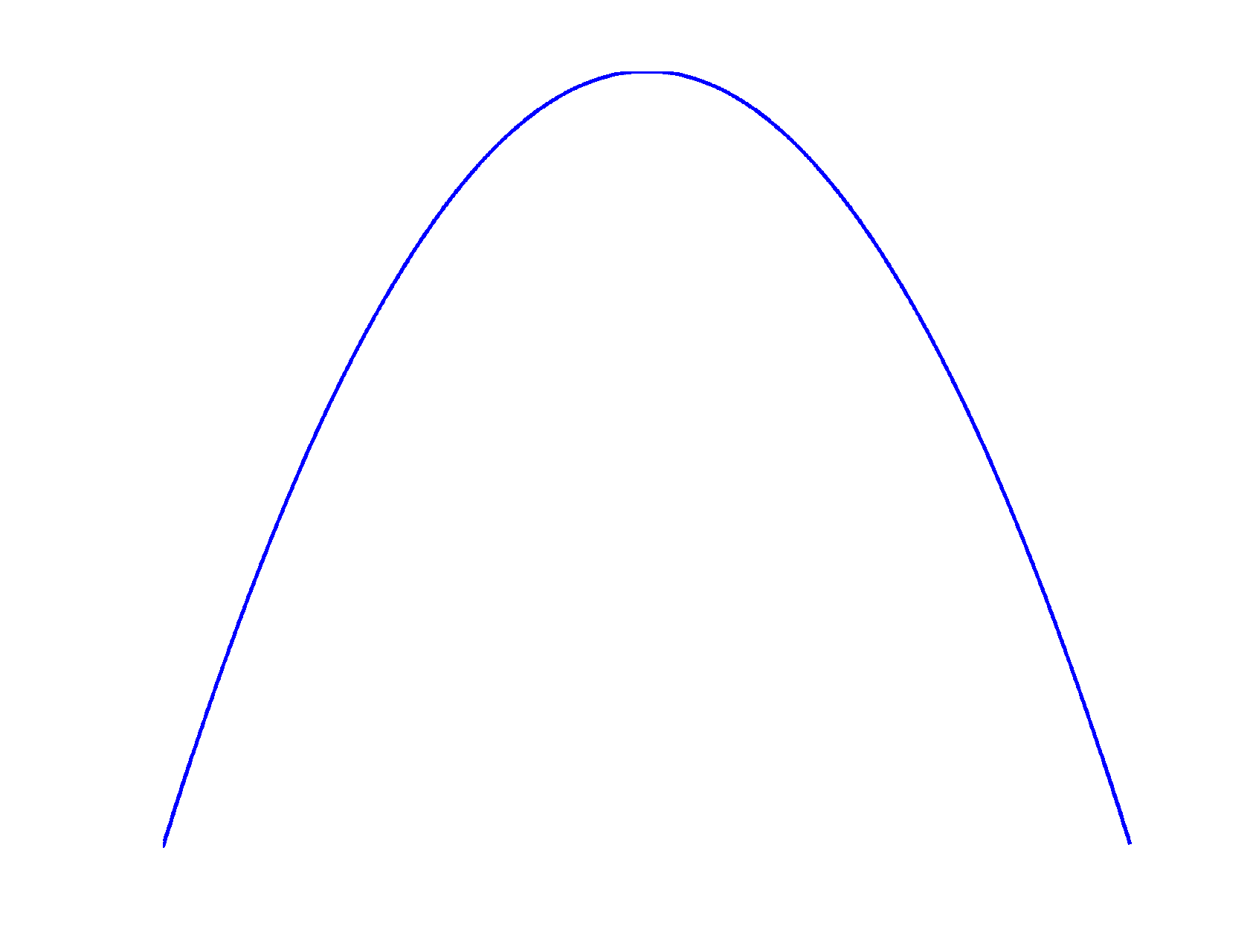}
                \caption{Takagi-Landsberg curves with different parameters: $a=-1/2$ (the alternating sign Takagi curve), $a=1/2$ (the Takagi-Blancmange curve), $a=2/3$, $a=1/4$ (parabola).  }
                \label{fig:Odometer}
\end{figure}

It  immediately follows from \eqref{eq:TakagiLandsberg} that  $\mathcal{T}_a$ satisfies the following de Rham type functional equations for $x\in[0,1]$:

\begin{equation}\begin{cases}\mathcal{T}_a(x/2) =  a \mathcal{T}_a(x)+x/2,\\
\mathcal{T}_a(\frac{x+1}{2}) =  a \mathcal{T}_a(x)+\frac{1-x}{2}.\end{cases}\label{eq:deRhamForTa}\end{equation}

Conversely,  as shown in \cite{Barnsley} using the Banach's fixed point theorem (see also \cite{Girgensohn1994} and \cite{Girgensohn2012}) any system of functional equations
\begin{equation}\begin{cases}f(x/2) =  a_0 f(x)+g_0(x),\\
f(\frac{x+1}{2}) =  a_1 f(x)+g_1(x),\end{cases}\label{eq:deRhamForTa}\end{equation}
with $\max\{|a_0|, |a_1|\}<1$ and such that consistency condition\footnote{i.e. both equations gives the same value at $x=1/2.$} holds
\begin{equation}
a_0\frac{g_1(1)}{1-a_1}+g_0(1)=a_1\frac{g_0(0)}{1-a_0}+g_1(0),
\end{equation}
defines a unique continuous function on $[0,1]$.

\subsection{Dyadic odometer}

Consider $X = \mathbf{Z}_2 =\prod\limits_{0}^\infty\{0,1\}$, the compact additive dyadic group of dyadic integers with Haar measure $\mu$, and let $T : Tx = x + 1$ be the addition of unity. The dynamical system  $(X,T,\mu)$ is called the dyadic odometer. It is one of the simplest transformations in the rank one category. Space $X$ can be identified with a paths space of a simple Bratteli-Vershik diagram with only one vertex at each level $n, n=0,1,2\dots$, see Fig. \ref{fig:Odometer}.  A cylinder set $C = [c_1c_2\dots c_n] = \{x\in X|x_1=c_1, \dots, x_n=c_n\}$ of a rank $n$ is totally defined by a finite path from the origin to the  vertex $n$. Sets $\pi_n$ of linearly ordered  (this order is  called adic or colexicographical, see \cite{Ver81} and \cite{ver82} for the original definition) finite paths are in one to one correspondence with towers
$\tau_n$ made up of corresponding cylinder sets.  Towers define  approximation of transformation $T$, see \cite{Ver81}. Cylinder sets $[x_1,x_2\dots, x_n],x_i\in\{0,1\},$ constituting towers are called rungs; the bottom rung corresponds to the cylinder  $[0,0,\dots,0].$   There are a total of $2^n$ paths in each  $\pi_n$ (or rungs in $\tau_n$).

%Let  $\pi_n$ denote the set of finite paths ${(x_1,x_2,\dots, x_n)}$ leading from the origin to the
 %level $n$ vertex. Odometer maps each nonmaximal path of length $n$ to its successor. We
  %consider $\pi_n$ as linearly  ordered sets $(\pi_n,\succeq)$ with the adic (co-lexicographical) order.

Using canonical mapping $\num: N\rightarrow \mathbb{N}_0$ defined by $\num(x) = \sum\limits_{i=1}^\infty x_i 2^i$, the set $N=\{x\in  \mathbf{Z}_2, \text{s.t.} \sum\limits_{i=1}^\infty x_i<\infty\}$ can be naturally identified with  nonnegative integers $\mathbb{N}_0$.  We follow the agreement that finite paths can be continued with zeroes, then the image set  $\num(\pi_n)$   is the discrete interval $[0,1,2,\dots, 2^n-1]$. Of course, for $x\in N$ we have $\num(T x) = \num(x)+1$.

\begin{figure}[t]
                \centering
                \includegraphics[trim=1cm 0cm 5cm 0cm,scale=0.75]{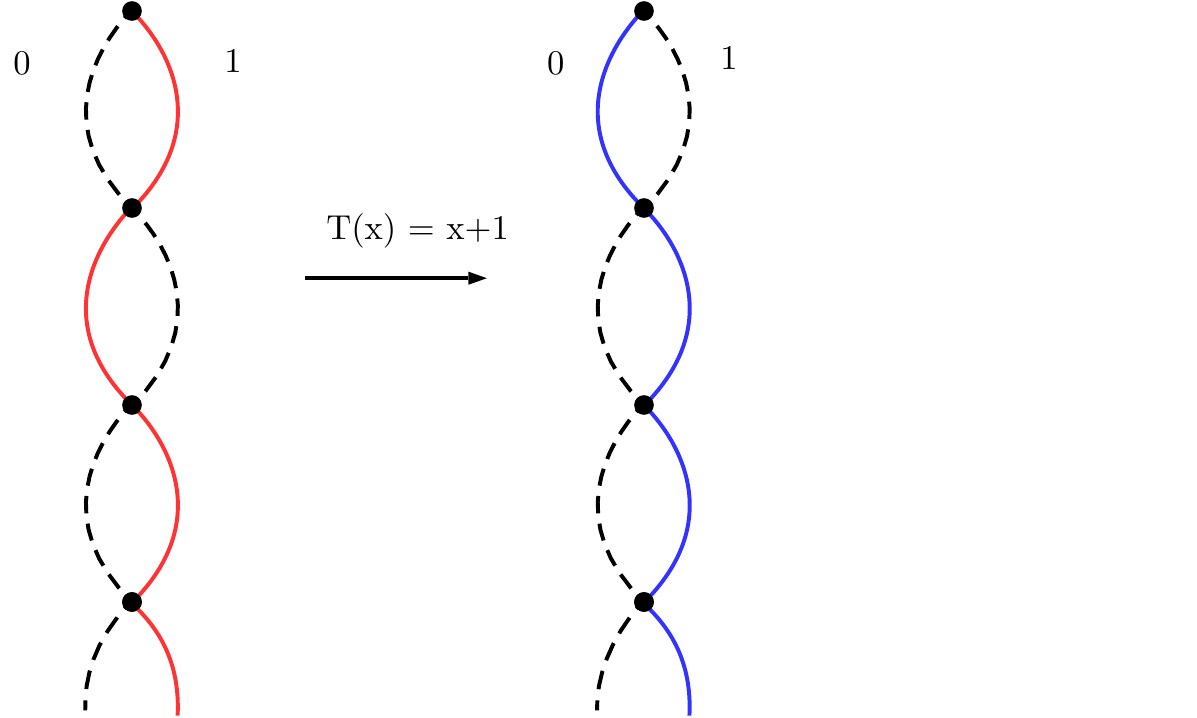}
                \caption{The dyadic odometer: the red (solid) path is mapped to the blue (solid) path.  }
                \label{fig:Odometer}
\end{figure}

The following lemma reads that for a.e. $x\in X$ we can choose a sequence of levels in the Bratteli diagram such that $x$ lies $\epsilon$-close to the bottom rung of $\tau_{n_j}, j=1,2\dots$. Its proof will follow the ideas of Janvresse and de la Rue from \cite{PascalLooselyBernoulli}.

\begin{lemma} For $\mu$-a.e. $x$ for any $\varepsilon >0$ there is a sequence $(n_j)_j$ such that index $\num(\omega^j)$ of a finite path $\omega^j=(x_1, x_2\dots, x_{n_j})$ satisfies the following inequality $\num(\omega^j)/|\pi_{n_j}|<\varepsilon$.
\label{lemma:bottomRung}
 \end{lemma}
\begin{proof} Let  $Z_m=2\sum_{i=1}^m(x_i-\frac12)$.
 Then  $Z_m$ is a symmetric random walk and thus recurrent. Therefore for $\mu$-a.e. $x$ and any $r\in\mathbb{N}$  we can choose a sequence of moments  $m_j$ such that $Z_{m_j+1}=Z_{m_j+2}=\dots=Z_{m_j+r}=-1.$ That means that $x$ belongs to the one of $2^{m_j}$  first rungs of the tower $\tau_{m_j+r}.$ Other way to say this is that $\num(x_1,x_2\dots,x_{m_j+r})<2^{m^j}$. Having in mind $|\tau_{n}|=2^n$ the former inequality can be rewritten as \[\frac{\num(x_1,x_2\dots,x_{m_j+r})}{|\tau_{m_j+r}|}<2^{-r}.\]
Choosing $r$ such that $2^{-r}<\varepsilon$ and setting $n_j=m_j+r$ we obtain the required statement.
\end{proof}

In  \cite{DeLaRue} and \cite{Min17} it was shown that the necessary condition for a limiting curve to exist is an unbounded growth of the normalizing coefficient $R_n$.  For the stationary odometer (unlike for the Pascal adic) this implies that there are no limiting curves for a \emph{cylindric} (i.e.  depending on a finite number of coordinates) function $g$ since partial sums $\mathcal{S}_{x}^g$ are a.s. bounded for such $g$, see \cite{Min17} Theorems 1-3. Therefore we need to consider noncylindric functions. Let $0<|q|<1$ we denote by $s_q:\mathbf{Z}_2\rightarrow \mathbb{R}$ a  weighted-sum-of-digits function defined by the identity
$s_q(x) =\sum\limits_{i=1}^\infty x_nq^n $ for $x=(x_1, x_2, \dots, x_n,\dots)\in \mathbf{Z}_2 .$ We denote by $\mathbf{0}$ the zero path $(0,0,\dots)$ and by $\mathcal{S}_{x}^q(n)$ the partial sums defined by $\mathcal{S}_{x}^{s_q}=\sum\limits_{j=0}^{n-1} s_q(T^jx)$.

\subsection{The Trollope-Delange formula}

For $q=1$ function $s_1:N\rightarrow\mathbb{R}$ is the well-known sum-of-binary-digits function. Commonly sum-of-binary-digits function  is denoted simply by $\textbf{s}$ instead of $s_1$  and  defined on $\mathbb{N}_0$ instead of $N$, but, as mentioned above, we identify sets $N$ and $\mathbb{N}_0$ and we have the identity $\textbf{s}\circ \num = s_1$. For $S=\mathcal{S}^1_{\mathbf{0}}$ the following  Trollope-Delange formula holds  (see \cite{Trollope} and \cite{Delange}):

\begin{equation}\label{eq:TrollopeDelange}
\frac{1}{n}S(n) = \frac12\log_2(n)+\frac12\tilde{F}(\log_2(n)),\end{equation}
where the $1$-periodic function $\tilde{F}$ is given by
\begin{equation}\tilde{F}(t) = 1-t - 2^{1-t}\mathcal{T}(2^{-(1-t)}), \text { for } 0\leq t\leq 1.\end{equation}

Several extensions of the Trollope-Delange formula are known. Besides the classical case $S(n)=\sum\limits_{j=0}^{n-1}\textbf{s}(j)$, its analogues for exponential $\sum\limits_{j=0}^{n-1}\exp(t\, \textbf{s}(j))$, power  $\sum\limits_{j=0}^{n-1} \textbf{s}^{k-1}(j)$ and binomial  $\sum\limits_{j=0}^{n-1} \binom{\textbf{s}(j)}{m}$  sums were studied, see \cite{Kruppel}, \cite{Girgensohn2012} or \cite{AllaartKawamura2012} for the history and starting links.

The Trollope-Delange formula also appears to be useful for describing functions like  $\varphi_n(t) = \frac{\mathcal{S}(t\cdot l_n(x)) - t \cdot \mathcal{S}(l_n)}{R_{n}}$. We  show below, see Proposition \ref{Prop:MainProp}, that  $\varphi_n(t_j)=\mathcal{T}(t_j)$ for $t_j=\frac{j}{l_n}, j=0,1,2\dots l_n,$ and $l_n=2^n$. Unfortunately, function $s_1$ is not well-defined on $X=\mathbf{Z}_2 $. So we find generalization of the Trollope-Delange formula for the weighted-sum- of-binary-digits function  $s_q$.

%Several approaches could adapted to solve this problem
%The method proceeds by extracting
%functional equations from the digital sum sequences such as S(n),
% Technically, we follow approach by Girgensohn as, aprioiri.

%We follow reasoning by Girgensohn. Doesn't need . Functional equations..

\section{Main results}

The main results of the paper are the following two statements:

\begin{theorem}
\label{Th:mainTheorem} Let $(X,T,\mu)$ be the dyadic odometer and $\frac12<|q|<1$.
Then for $\mu$-a.e. $x$ there exists a stabilizing sequence $l_n=l_n(x)$ such that $\varphi_{x,l_n}^{s_q}$ converges in sup-metric on $[0,1]$ to the  function $-\mathcal{T}_{a}$, where $a={1}/(2q)$.
\end{theorem}

For a nonnegative integer $j$ with binary expansion  $j = \sum\limits_{i\geq 0}\omega_i2^i$ we denote by $\textbf{s}_q(j)$ the weighted-sum-of-binary-digits function $\sum\limits_{i\geq 0}\omega_iq^{i+1}$ and set $S_q(n) =\sum\limits_{j=0}^{n-1}\textbf{s}_q(j)$.
The next proposition does not use the notion of the odometer and its first part  generalizes Theorem $5.1$ by Kr\"uppel in \cite{Kruppel2008}, where author considers the case of "alternating sums"\,, that are obtained here  if $q=-1$.

\begin{proposition}
\label{Prop:MainProp}
Let $|q|>1/2$ and $a=1/(2q)$.

$1.$  The following generalized Trollope-Delange formula holds

\begin{equation}
\label{eq:generalizedTrollopeDelange}
\frac1nS_q(n) = \frac{q}{2}\bigg(\frac{1-q^{\log_2(n)}}{1-q} + q^{\log_2(n)}\hat{F}_q(\log_2(n))\bigg), \end{equation}
where the $1$-periodic function $\hat{F}_q$ is given by $\hat{F}_q(u) = \frac{1-q^{1-u}}{1-q}-q^{-u}2^{1-u}\mathcal{T}_a(2^{-(1-u)}), $ $u\in[0,1].$

$2.$ Let $l=2^k$ for any fixed $k\in \mathbb{N}$, then the following identity holds:

\begin{equation}{\varphi}_l^q(t_j) = -q \mathcal{T}_a(t_j),\label{eq:varphi_l}\end{equation} %aka tilde G
where $\varphi_l^q(t)= \frac{S_q(t\cdot l) - t \cdot S_q(l)}{R_{l}}$ $t_j=\frac{j}{l}, j=0,1,2,\dots,l.$ Moreover, the renormalization coefficient $R_{l}$ is proportional to $(2|q|)^{\log_2(l)}$.
\end{proposition}

\emph{Remark} $1.$ Conditions $\frac12<|q|<1$ and $|q|>1/2$ in Theorem \ref{Th:mainTheorem} and Proposition \ref{Prop:MainProp} respectively are essential. For $|q|\leq1/2$  no continuous limiting curve exists.

\emph{Remark}  $2.$ Another  generalization of the Trollope-Delange formula for weighted-sum-of-binary-digits functions  was obtained in \cite{Larcher-et-al-2005} and \cite{Hofer-et-al-2008}. Their approach is different and result is asymptotic for our choice of weights $(q^i)_{i=1}^\infty$, in formula \eqref{eq:generalizedTrollopeDelange} we give a precise expression.

\subsection{Proof of Theorem \ref{Th:mainTheorem}}
The result of Theorem \ref{Th:mainTheorem} follows from the second assertion of the Proposition~ \ref{Prop:MainProp} and Lemma~\ref{lemma:bottomRung}.

\begin{proof}
Clearly we have the identities $S_q(n)=\mathcal{S}_{\mathbf{0}}^q(n)$ and $\varphi_{\mathbf{0},l}^{s_q}=\varphi_{l}^q$.
 Lemma~\ref{lemma:bottomRung} implies that  for $r=r_j$ arbitrary large (we skip index $j\in\mathbb{N}$ below)   we can assume that $x = (\underbrace{x_1,x_2, \dots ,x_m,}_m \underbrace{0, 0\dots, 0}_r, *, *, \dots,*, \dots),$ where  $m=m(r)$ and symbol $*$  means  either $0$ or $1$ entry.
Let $y=(\underbrace{x_1,x_2, \dots ,x_m,}_m \underbrace{0, 0\dots, 0}_r, 0\dots)$. We have  \[\frac{|\mathcal{S}_{y,q}(i)-\mathcal{S}_{\mathbf{0},q}(i)|}{R_n}\leq \frac{2^{m+1}\sum\limits_{i=1}^{2^m}q^i}{c2^{m+r}} \xrightarrow[r\rightarrow\infty]{} 0,\ i=0,\dots, 2^{m+r}-2^m-1.\]
and, therefore, $\parallel\varphi_{y,l_j}^{s_q} -\varphi_{\mathbf{0},l_j}^{s_q}\parallel_{\infty}\xrightarrow[j\rightarrow\infty ]{} 0,\ l_j=2^{m(r_j)+r_j}.$
Next note, that tail coordinates of $x$ denoted by $(*, *, \dots,*, \dots)$ above do not change under $T^i, i=0,\dots, 2^{m+r}-2^m-1.$ Namely, we have $s_q(T^i x) = s_q(T^iy)+b,$ where $b=b(r)$ is constant in $i$, $i=0,\dots, 2^{m+r}-2^m-1$. Due to the renormalization\footnote{Specifically,
 subtraction of the linear part $ t \cdot \mathcal{S}^g_x(l_n)$.}, we have $\varphi_{x,n}^{g+C}=\varphi_{x,n}^{g}$ for any constant $C$, so the value of $b$ does not affect\footnote{Equivalently, we may assume that $b=b_j=0$ for the given $l_j$.} $\varphi_{x,l_j}^{s_q}$.
Finally, we conclude that our choice of the stabilizing sequence $l_{j}(x)$ implies $||\varphi_{x,l_j}^{s_q} -\varphi_{\mathbf{0},l_j}^{s_q}||_{\infty}\xrightarrow[j\rightarrow \infty]{} 0$.
\end{proof}

\subsection{Proof of Proposition \ref{Prop:MainProp}}

Several approaches can be used to prove \eqref{eq:generalizedTrollopeDelange}.
We adopt approach suggested by Girgensohn in \cite{Girgensohn2012}. The idea of the approach is to start with the sequence $S(n)$ itself, discover the functional
equations within this sequence and then identify the limiting functions from the obtained functional
equations. Advantage of this approach is that it  does not require any advance knowledge of the functions appearing in the answer.

\begin{proof}
First we note\footnote{Identities \eqref{eq:s(2j)} and \eqref{eq:s(2j+1)} can be generalized for $\omega\in \mathbf{Z}_2$ and $s_q$ function using the shift operator but we do not use that. } that for any $p=2^{k-1}$
\begin{align}
\label{eq:s(2j)}
&\mathbf{s}_q(2j) = q\, \mathbf{s}_q(j),\\
\label{eq:s(2j+1)}
&\mathbf{s}_q(2j+1) =q\, \mathbf{s}_q(j)+q,\\
\label{eq:s(j+p)}
&\mathbf{s}_q(j+p) = \mathbf{s}_q(j)+q^k, \quad j=0,1,\dots, p-1,\\
\label{eq:s(j+p)}
&\mathbf{s}_q(j+p) = \mathbf{s}_q(j)-q^k(1-q), \quad j=p,p+1,\dots, 2p-1.\end{align}

Let $k_n = [\log_2(n)]$ and $u_n = \{\log_2(n)\}$, where $[\cdot]$ and $\{\cdot\}$ stand for the integer and fractional parts respectively. Following \cite{Girgensohn2012}  we denote by $p_n=p(n)=2^{k_n}$ the largest power of $2$ less than or equal to $n$ and by $r_n$ we denote $q^{\log_2(p_n)}=q^{k_n}$.

For any $n\in\mathbb{N}$ we have
\begin{align}
&S_q(n+2p_n) = S_q(n)+S_q(2p_n)+nq^{k_n+1},\label{eq:Sn+2p}\\
&S_q(n+p_n) = S_q(n)+(2q-1)S_q(p_n) - (n-p_n)q^{k_n}(1-q)+q p_n,\label{eq:Sn2+p}\\
&S_q(2n) = 2q S_q(n)+nq. \label{eq:S2n}\end{align}

 It is straightforward  to obtain from \eqref{eq:S2n} that for any $p=2^k$
\begin{equation}S_q(p) = q\frac{1-q^{k-1}}{1-q}2^{k-1} = q\frac{1-q^{k-1}}{1-q}\frac{p}{2}.\end{equation}

We define function $G_q(n)$ by the identity \begin{equation}G_q(n) = \frac{1}{p(n)r_n}\Big(S(n)-\frac{n}{p(n)}S(p_n)\Big).\end{equation}

\begin{figure}[t]
                \centering
                \includegraphics[scale=0.45]{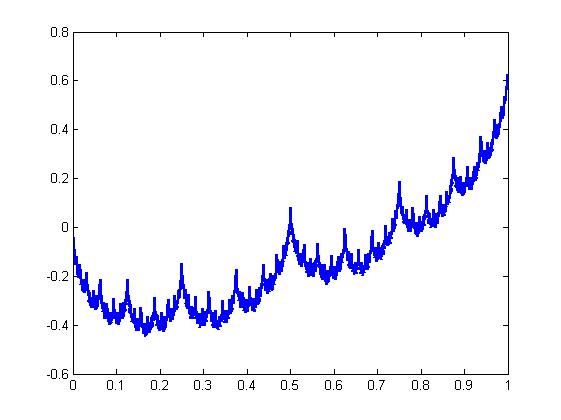}
                \caption{Graph of the function $F_{2/3}$.}
                \label{fig:Odometer}
\end{figure}

Function $G_q(n)$ satisfies the following three identities:
\begin{align}
&G_q(2n) = G_q(n),\label{eq:G(2n)}\\
&G_q(n+p_n) = \frac{1}{2q}G_q(n)+\frac{p_n-n}{4p_n}\big(3-2q\big),\label{eq:G(n+pn)}\\
&G_q(n+2p_n) =\frac{1}{2q}G_q(n)+\frac{n}{4p_n}\big(2q-1\big).\label{eq:G(n+2pn)}
\end{align}

Identities \eqref{eq:G(2n)}--\eqref{eq:G(n+2pn)} follow from \eqref{eq:Sn+2p}--\eqref{eq:S2n}; here we deduce the first one, other two are obtained in the same way.
\[\begin{split} &G_q(n+p_n) = \frac{1}{2qp_nr_n}\Big(S(n+p_n)-\frac{n+p_n}{p(n+p_n)}S(p(n+p_n))\Big) =\\
&=\frac{1}{2q}G_q(n) +\frac{1}{2qp_nr_n}\Bigg(\frac{p_n-n}{2}q+(1-q)\bigg(\Big(\frac{n}{p_n}-1\Big)S(p_n)+(n-p_n)q^{k_n}\bigg)\Bigg) = \\
&=\frac{1}{2q}G_q(n)+\frac{p_n-n}{4p_n}\big(3-2q\big).\end{split}\]

Following \cite{Girgensohn2012}, we set $x_n= x(n)= 2^{u_n}-1= \frac{n-p_n}{p_n}\in[0,1]$,
then the following simple  lemma proved in \cite{Girgensohn2012} holds:

\begin{lemma}
\label{lemma:Girgensohn}
Let $G:\mathbb{N}\rightarrow \mathbb{R}$ be a function on the integers. For $n\in \mathbb{N}, $ set \[x:=\frac{n-p_n}{p_n}\in[0,1) \quad \text{and} \quad F(x) =F\bigg(\frac{n-p_n}{p_n}\bigg):=G(n). \]
Then $F$ is a well-defined function on the dyadic rationals in $[0,1)$ iff $G(2n)=G(n)$ for all $n\in \mathbb{N}.$
\end{lemma}
Also we have
\begin{align}
&\frac{x(n)}{2}=\frac{n-p_n}{2p_n}=x(n+p_n),\\
&\frac{x(n)+1}{2}=\frac{n}{2p_n}=x(n+2p_n).
\end{align}
By  Lemma \ref{lemma:Girgensohn} the  function $F_q$ given by $F_q(x_n) =G_q(n)$ is well-defined on the dyadic rationales in $[0,1)$.
Identities \eqref{eq:G(n+pn)}--\eqref{eq:G(n+2pn)}  for $x=x_n$ rewrites as follows
\begin{equation}\begin{cases}F_q(x/2) =  a F_q(x)+(2q-3)\frac{x}{4},\\
F_q(\frac{x+1}{2}) =  a F_q(x)+(2q-1)\frac{x+1}{4}.\end{cases}\label{eq:FbydeRham}\end{equation}

From the system \eqref{eq:FbydeRham} we obtain
$F_q(x) =qx-\frac12\mathcal{T}_{a}(x) .$
Using $S_q(n) = r_np_n G_q(n)+\frac{n}{p_n}S_q(p_n)$ we see that
\begin{equation}
\label{eq:generalizedTrollopeDelange2}
\frac{1}{n}S_q(n) =q^{k_n} \frac{F_q(x_n)}{x_n+1}+\frac{q}{2}\frac{1-q^{k_n}}{1-q}.\end{equation}

Using $x_n=2^{u_n}-1$ and rewriting $F_q$ as $F_q(x)=  q\frac{x+1}{2}-q\mathcal{T}_a(\frac{x+1}{2})$
we obtain~\eqref{eq:TrollopeDelange}. This finishes part $1$ of the proof.

Now we prove part $2$. We fix some $l=2^j, j\in\mathbb{N}.$  We set $j_m= j-m, m=1,2,\dots,j $ and $l_m = l/2^m$.  We split the whole discrete interval $I = [1, l)$ into $j$ disjoint intervals $(I_m)_{m=1}^j,$ $I_m=[l_m, l_{m-1})$, thus  $I=\cup_m I_m.$

Let $r=q^{j-1}$ and $p=l/2$, importantly $r$, $p$,  and $j_m$ do not change in $n$ now, they depend only on $l$ and $m$.  We set $\tilde{G}(n) = \frac{1}{rp}\Big( S(n)-\frac{n}{2p}S(2p)\Big)\equiv\varphi_l(\frac{n}{l}).$
We prove \eqref{eq:varphi_l} at each $I_m,m=1,2,\dots,j,$ by starting with $I_1$ and then going to the general case $I_m, m \geq 1$.

For $I_1$  we have $\frac{l}{2}\leq n <l$ and at this interval $p=p_n $, $r=r_n$ and $\frac{x+1}{2} =\frac{n}{l}$, where $p_n,r_n $ and $x=x_n$ are defined in the proof of part $1$. Using \eqref{eq:Sn+2p}-\eqref{eq:Sn2+p} we rewrite $\tilde{G}(n) $ as follows:
\[\tilde{G_q}(n) = \frac{1}{pr}\Big(S_q(n)-\frac{nq}{2p}(2S_q(p)+p)\Big) =  \frac{1}{pr}\Big(S_q(n)-\frac{n}{p}S_q(p)+(1-q)\frac{n}{p}S_q(p)-q\frac{n}{2}\Big) =\]
\[=G_q(n)-\frac{nq}{2pr}+(1-q)\frac{n}{p}\frac{S_q(p)}{pr} = G_q(n)-\frac{nq}{2pr}+\frac{nq(1-q^{j-1})}{2pr} = G_q(n)-\frac{qn}{2p}=\]
\[=qx-\frac12\mathcal{T}_a(x)-q\frac{x+1}{2} = q\Big(\frac{x-1}{2}-\frac{1}{2q}\mathcal{T}_a(x)\Big) = -q\mathcal{T}_a\big(\frac{x+1}{2}\big) = -q\mathcal{T}_a\big(\frac{n}{l}\big),\]

Generally, we define $l_m\leq n<l_{m-1}$ and set $t =\frac{x_n+1}{2}$, note that $\frac{t}{2^{m-1}}=\frac{n}{l}$. From \eqref{eq:S2n} and \eqref{eq:deRhamForTa} we see that

\begin{equation}
S_q(l) = (2q)^m S_q\Big(\frac{l}{2^m}\Big) + 2^{m-1}l_mq(1+q+\dots+q^{m-1})
\label{eq:Sl}
\end{equation}
and
\begin{equation}
\mathcal{T}_a(t) = (2q)^m \mathcal{T}_a\Big(\frac{t}{2^m}\Big) -t q(1+q+\dots+q^{m-1}).
\label{eq:Tl}
\end{equation}

Using \eqref{eq:Sn+2p}-\eqref{eq:Sn2+p}  and \eqref{eq:Sl}-\eqref{eq:Tl} on  the interval $I_m$ we have $l_m\leq n <l_{m-1}$\\ and
$\tilde{G_q}(n) = \frac{2}{lr_{j_1}}\Big(S_q(n) -\frac{n}{l}S_q(l)\Big) = \frac{2}{lr_{j_1}}\Big(S_q(n) - \frac{n}{l_{m-1}}S_q(l_{m-1})+\frac{n}{l_{m-1}}S_q(l_{m-1})-\frac{n}{2^{m-1}l_{m-1}}S_q(l)\Big)$
$=\frac{2}{lr_{j_1}}\bigg(r_{j-m+1}\frac{l_m}{2}\tilde{G_q}(n)+\frac{n}{l_m}S_q(l_{m-1})-\frac{n}{2^{m-1}l_{m-1}}\Big((2q)^{m-1}S_q(\frac{l}{2^{m-1}})+2^{m-2}l_{m-1}(q+\dots+q^{m-1})\Big)\bigg)$
$=\frac{1}{(2q)^{m-1}}\tilde{G_q}(n)+\frac{n}{l_{m-1}}\frac{2}{lr_{j_1}}\Big((1-q^{m-1})S_q\big(\frac{l}{2^{m-1}}\big)+\frac{2^{m-2}l_{m-1}}{2^{m-1}}(q+\dots+q^{m-1})\Big)=$
$-\frac{q}{(2q)^{m-1}}\mathcal{T}_a(t)+\frac{n}{lr_{j_1}}\Big(q(1+\dots+q^{m-2})(1-q^{j_{m-1}-1})-(q+\dots+q^{m-1}))\Big)$
$=-\frac{q}{(2q)^{m-1}}\mathcal{T}_a(t)-\frac{n}{lr_{j_1}}q^{j_{m-1}-1}(q+\dots q^{m-1})=-q\mathcal{T}_a(t/2^{m-1})+\frac{t(q+\dots+q^{m-1})}{(2q)^{m-1}}-t\frac{1}{(2q)^{m-1}}(q+\dots+q^{m-1})=-q\mathcal{T}_a\big(\frac{n}{2^l}\big).$

%This finishes part $2$ of the proof.
\end{proof}

\emph{Question.} One can consider a more genaral function $\tilde{s}_2(\omega)$ defined by $\tilde{s}_2(\omega) = \sum\limits_{i=1}^\infty c_i \omega_i$ with $\sum\limits_{i=1}^\infty |c_i| <\infty$. What are the limiting curves in this case?

\textbf{Acknowledgements.} Author is grateful to Andrey Lodkin and Ilya Manaev for fruitful discussions of the ideas of the paper.

\end{document}